\documentclass[12pt]{article}
\usepackage{amsmath,amssymb,amsbsy,amsfonts,amsthm,latexsym,amsopn,amstext,amsxtra,euscript,amscd}
\usepackage{datetime}
\newtheorem{theorem}{Theorem}

\newtheorem{cor}[theorem]{Corollary}

\newtheorem{thm}{Theorem}
\newtheorem{exam}{Example}
\newtheorem{example}[exam]{Example}
\newtheorem{conj}{Conjecture}
\newtheorem{conjecture}[conj]{Conjecture}
\newtheorem{lem}{Lemma}
\newtheorem{lemma}[lem]{Lemma}

\usepackage[np]{numprint}
\npdecimalsign{\ensuremath{.}}

\DeclareMathOperator{\Li}{Li}

\def\\{\cr}
\def\({\left(}
\def\){\right)}
\def\[{\left[}
\def\]{\right]}
\def\<{\langle}
\def\>{\rangle}
\def\fl#1{\left\lfloor#1\right\rfloor}

\def\N{\mathbb{N}}
\def\C{\mathbb{C}}

\def\R{\mathbb{R}}

\def\Z{\mathbb{Z}}

\def\id{\mathrm{id}}

\def\notdivides{\mathrel{\kern-3pt\not\!\kern3.5pt\bigm|}}

\begin{document}

\title{\textbf{Sparse sets that satisfy the Prime Number Theorem}}



\author{
{\sc Olivier Bordell\`{e}s}  \\
2 All\'ee de la combe,  \\
43000 Aiguilhe, France\\
{borde43@wanadoo.fr}
\and
{\sc Randell Heyman}\\
{School of Mathematics and Statistics} \\
{University of New South Wales}\\
{Sydney, NSW 2052, Australia}\\
{randell@unsw.edu.au}
\and
{\sc Dion Nikolic}\\
{University of New South Wales} \\
{Canberra, ACT 2612, Australia}\\
{d.nikolic@ad.unsw.edu.au}
}

\date{}
\maketitle

\begin{abstract}
For arbitrary real $t>1$ we examine the set $\{\fl{x/n^t}:n \le x\}$. Asymptotic formulas for the cardinality of this set and the number of primes in this set are given. The prime counting result uses an alternate Vaughan's decomposition for the von Mangoldt function, with triple exponential sums instead of double exponential sums. These sets are the sparsest known sets that satisfy the prime number theorem, in the sense that the number of primes is asymptotically given by the cardinality of the set divided by the natural logarithm of the cardinality of the set.

\end{abstract}

\noindent
\newline
\newline
\section{Introduction}
Let $x$ and $n$ be strictly positive integers throughout. The prime number theorem states that $\pi(x) \sim x/\log x$, where $\pi(x)$ is the number of primes less than or equal to $x$. We can apply the prime number theorem to sets in the following way:
Let
$$S_{f,x}:=\{f(n):n \le x\},$$
where we restrict $S_{f,x}$ to those sets where $f$ is positive integer valued and not piecewise defined. We will count the number of primes in these sets using
$$\pi\(S_{f,x}\):=\left|\{s \in S_{f,x}: s\text{ is prime}\}\right|.$$
Using this notation the prime number theorem can be restated as
$$\pi\(S_{\id,x}\)\sim \frac{|S_{\id,x}|}{\log |S_{\id,x}|},$$
where  $|*|$ represents set cardinality and $\id$ is the identity function.
We say a set $S_{f,x}$ \textit{satisfies the prime number theorem} if
$$\pi\(S_{f,x}\)\sim \frac{|S_{f,x}|}{\log |S_{f,x}|}.$$

For any real number $m$, denote by $\fl{m}$  the floor of (or more formally the largest integer not exceeding) $m$. The floor function appears prominently in Beatty sequences (see, for example, ~\cite{ABS,BaBa,BaLi,GuNe,Harm}) and Piatestski-Shapiro sequences (see, for example,~\cite{Akb,BBBSW,BBGY,BGS,LSZ,Morg}), both which we will consider in this paper.

Recently there has been interest in the floor function set $S_{\fl{\frac{x}{n}},x}$. That is, the set
$$\left\{\fl{\frac{x}{n}}: n \le x\right\}.$$ This set arose naturally from recent papers estimating sums of the form $\sum_{n \le x} f\(\fl{\frac{x}{n}}\)$ (see, for example, \cite{Bor2,Che, Liu,Ma,Ma2,Stu,Wu,Wu2,Zha,Zha2,Zhao}).

From  \cite{Hey} we have
\begin{align}
\label{eq:cardinality F(x)}
\left|S_{\fl{\frac{x}{n}},x}\right|&=\fl{\sqrt{4x+1}}-1=2\sqrt{x}+O(1).
\end{align}
From \cite{Hey2} we have
$$\pi\(S_{\fl{\frac{x}{n}},x}\)\sim\frac{4 \sqrt{x}}{\log x}.$$
An asymptotic formula with error term is also now known (see \cite{Ma3}) as follows:
\begin{align}
\label{eq:primes in F(x)}
\pi\(S_{\fl{\frac{x}{n}},x}\)&=\frac{4 \sqrt{x}}{\log x}+O\(\sqrt{x}e^{-c(\log x)^{3/5}(\log \log x)^{-1/5}}\),
\end{align}
where $c>0$ is a positive constant.
In \cite{Ma3}, and in correspondence with the authors, Wu and Ma made the astute observation that $S_{\fl{\frac{x}{n}},x}$ is a sparse set that satisfies the prime number theorem. That is,
$$\pi\(S_{\fl{\frac{x}{n}},x}\)\sim \frac{4\sqrt{x}}{\log x}\sim\frac{\left|S_{\fl{\frac{x}{n}},x}\right|}{\log \left|S_{\fl{\frac{x}{n}},x}\right|}.$$
The authors are not aware of literature that focuses on  sets that satisfy the prime number theorem. As prime numbers are an essential component of many cryptography protocols, it is conceivable that sets that satisfy the prime number theorem (particularly those less dense) may have applications.

To explore this further, define the \textit{density} of a set $S_{f,x}$, if it exists, as
$$D_{f,x}:=\frac{|S_{f,x}|}{\max{S_{f,x}}-\min{S_{f,x}}}.$$
We say a function $f$ (or its set $S_{f,x}$) is \textit{sparse} if $D_{f,x} \not \sim 1$. A function $f$ (or its set $S_{f,x}$) is said to be \textit{sparser} than a function $g$ (or its set $S_{g,x}$) if $\lim_{x \rightarrow \infty} D_{f,x}/D_{g,x}=0$.

As mentioned above, the set $S_{id,x}$ satisfies the prime number theorem. But it is not sparse.
In contrast, as implied from above, $D_{\fl{x/n},x}\sim  2 x^{-1/2}$  and is an example of a sparse set that satisfies the prime number theorem. This paper investigates other sparse sets that satisfy the prime number theorem.

Using \cite{Guo} and \cite{RivW}, we briefly review 2 families of sets based on well-known sequences involving the floor function.
For fixed real numbers $\alpha,\beta$ the \textit{Beatty sequence} is the sequence $(\fl{\alpha n+\beta})_{n=1}^\infty$. It is known that there are infinitely many primes in the Beatty sequence if $\alpha \ge 1$. For our purposes we consider $\beta$ to be positive. The Piatetski-Shapiro sequence is the sequence $(\fl{n^c})_{n=1}^\infty$ with $c >1, c \not \in \N$. It is known that the sequence contains infinitely many primes for all $c \in \(1,\frac{243}{205}\)$. Families of sets based on these sequences satisfy or do not satisfy the prime number theorem as follows:
\begin{cor}
\label{thm:beatty}
With the constraints mentioned above, the family of sets
$S_{\fl{n^c},x}$
do not satisfy the prime number theorem.
In contrast, the family of sets $S_{\fl{\alpha n + b},x}$
do satisfy the prime number theorem.
Moreover,
$D_{\fl{\alpha x+\beta},x} \sim \frac1{\alpha}$.
\end{cor}
While $S_{\fl{x/n},x}$ and $S_{\fl{\alpha n + b},x}$ both satisfy the prime number theorem, the set  $S_{\fl{x/n},x}$ is sparser than any given set of the form $S_{\fl{\alpha n + b},x}$. Our main result is as follows:

\begin{thm}
\label{thm:x/n^t}
Let $t > 1$, $x \geqslant e$ be real numbers satisfying \eqref{eq:hyp_x}. Then, for all $N \in \Z_{\geqslant 0}$,
\begin{align*}
   \pi \left( S_{\left\lfloor\frac{x}{n^t}\right\rfloor,x}\right) & = \left( \frac{x}{t^t}\right)^{\frac{1}{t+1}} \frac{(t+1)^2}{\log x} \sum_{k=0}^{N} \left( \frac{t \log t}{\log x}\right)^k \\
   & \hspace*{1cm} + \left( \frac{x}{t^t}\right)^{\frac{1}{t+1}}  \frac{t+1}{\log x} \sum_{j=1}^N \sum_{k=0}^{N-j} j! {j+k \choose k} \frac{t^k(t+1)^j (\log t)^k}{(\log x)^{j+k}} \left( 1 + (-t)^j t\right) \\
   & \hspace*{2cm} + O_N \left( \left( \frac{x}{t^t}\right)^{\frac{1}{t+1}} \frac{t^2}{\log x} \left( \frac{t^2 \log t}{\log x} \right)^{N+1} \right).
\end{align*}

\end{thm}
The error terms in the asymptotic formulas reduce with increasing $N$, as illustrated below.
\begin{example}
\begin{enumerate}
   \item[]
   \item[\scriptsize $\triangleright$] with $N=0$ we derive
$$\pi \left( S_{\left\lfloor\frac{x}{n^t}\right\rfloor,x}\right) = \left( \frac{x}{t^t}\right)^{\frac{1}{t+1}} \frac{(t+1)^2}{\log x}  + O \left( \left( \frac{x}{t^t}\right)^{\frac{1}{t+1}} \frac{t^{4}\log t}{\left( \log x \right)^{2}} \right)$$
provided that $x \geqslant \max \Bigl( (368t)^{770(t+1)}, (\log x)^{308t(t+1)} , t^{2t^2},t^te^{\np{1636}} \Bigr)$.
   \item[\scriptsize $\triangleright$] with $N=1$ this gives
\begin{align*}\pi \left( S_{\left\lfloor\frac{x}{n^t}\right\rfloor,x}\right) &= \left( \frac{x}{t^t}\right)^{\frac{1}{t+1}} \frac{(t+1)^2}{\log x} \left( 1 + \frac{t \log t + 1 -t^2}{\log x}\right)\\
& \qquad + O \left( \left( \frac{x}{t^t}\right)^{\frac{1}{t+1}} \frac{t^{6}(\log t)^2}{\left( \log x \right)^{3}} \right),\end{align*}
provided that $x \geqslant \max \Bigl( (368t)^{\np{1155}(t+1)}, (\log x)^{462t(t+1)}, t^{2t^2} , t^t e^{\np{2454}} \Bigr)$.
\end{enumerate}
\end{example}
When considered with the cardinality of $S_{\fl{x/n},x}$ (See Lemma \ref{lem:card x/n^t}), Theorem \ref{thm:x/n^t} implies that the set $S_{\fl{\frac{x}{n^t}},x}$ satisfies the prime number theorem. Noting that
$\(D_{\fl{\frac{x}{n^t}},x}\)_{t=1}^\infty \rightarrow \frac1{x}$,
we conclude that, for large enough $t$, the set $S_{\fl{\frac{x}{n^t}},x}$ will be sparser than any set based on Beatty sequences.

Recall that throughout the function $f$ in $S_{f,x}$ is never piecewise defined.  We conjecture as follows:
\begin{conjecture}
The family of sets $S_{\fl{\frac{x}{n^t}},x}$ are the sparsest sets that satisfy the prime number theorem.
\end{conjecture}

Throughout $p$ indicates a prime number and $\N=\{1,2,,\ldots\}$. The statements $f(x)=O\(g(x)\)$ and $f(x) \ll g(x)$ denote that there exists a constant $c>0$ such that $|f(x)| \le c|g(x)|$ for large enough $x$. The statement $f(x) \sim g(x)$ is the assertion that $\lim_{x \rightarrow \infty} f(x)/g(x)$ equals 1. Also, $\Lambda(n)$ denotes the von Mangoldt function and $\mu(n)$ denotes the M\"obius function. We use $\psi(x)$ to denote $x-\fl{x}-1/2$. We denote with $e(x)$ the expression $e^{2\pi i x}$. Finally, for all arithmetic functions $f$ and $g$, the Dirichlet convolution $f*g$ of $f$ and $g$ is given by
$$(f*g)(n)=\sum_{d|n} f(d) g(n/d).$$

\section{Organisation of this Paper}

The proof of the main theorem (Theorem \ref{thm:x/n^t}) begins in Subsection \ref{SS:Preliminaries} where we give the split $$\pi\(S_{\fl{\frac{x}{n^t}},x}\)=S_1(x,t)+S_2(x,t)+S_3(x,t).$$

The sums $S_1(x,t)$ and $S_3(x,t)$ are estimated in Subsection \ref{SS:S1} and Subsection \ref{SS:S3} respectively.

The remaining sum, $S_2(x,t)$, is split into $\sum_{t,1}(x)$ and $\sum_{t,2}(x)$. The sum $\sum_{t,1}(x)$ is estimated in Subsubsection \ref{SSS:S1}. The estimation uses Lemmas \ref{le:bound_log}, \ref{le:expansion_log}, \ref{lem:sum_p_beta_inv_0} and \ref{lem:sum_p_beta_inv}. These four lemmas are found in Subsection \ref{SS:Basic tools}.

The sum $\sum_{t,2}(x)$ is estimated in Lemma \ref{lem:o(x)}. This lemma in turn uses Lemmas \ref{lem:Dirichlet_exponential_principle}, \ref{lem:mu}, \ref{lem:Vaughan_bis}, \ref{lem:sarrob04} (all found in Subsection \ref{SS:tech lem}) and Lemma \ref{lem:Lambda_exp_bord} (found in Subsection \ref{SS:exp sum prime}).

\section{Preparatory Lemmas}
\label{S:Prep Lemmas}
The following is a generalisation of \cite[Thm 2]{Hey}.
\begin{lem}
\label{lem:card x/n^t}
For all $t>1$ we have
$$\left|S_{\fl{\frac{x}{n^t}},x}\right|=x^{\frac1{t+1}}\(t^{-\frac{t}{t+1}}+t^{\frac1{t+1}}\)+O(1).$$
\end{lem}

\begin{proof}
The elements of $S_{\fl{x/n^t},x}$ can be read off the graph $f(n)=xn^{-t}$. Note that $f$ is strictly decreasing. Let $a$ be such that $f'(a)=-1$. Calculating the derivative and using basic algebra we see that $a=(tx)^{1/(t+1)}$.

First consider the small values of $S_{\fl{x/n^t},x}$.   As the slope of $f(n)$ is greater than -1 for all integers $n>a$, it follows that
$$\fl{x(n)^{-t}}-\fl{x(n+1)^{-t}}<1, \,\, n>a.$$
Therefore $0,1,\ldots, \fl{x/a^t} \in S_{\fl{x/n^t},x}$, which implies that the number of elements of $S_{\fl{x/n^t},x}$ less than or equal to $x/a^t$ is
\begin{align}
\label{eq:cards}
t^{-\frac{t}{t+1}}x^{\frac1{t+1}}+O(1).
\end{align}

We now consider the large elements of $S_{\fl{x/n^t},x}$. That is, elements of $S_t(x)$ greater than or equal to $x/a^t$. Since the slope of $f(n)$ is less than -1 for $n < a$ it follows that each $n\in\{1,2,\ldots, \fl{(tx)^{1/(t+1)}}\}$ generates a unique element of $S_{\fl{x/n^t},x}$.
Therefore the number of elements of $S_{\fl{x/n^t},x}$ greater than or equal to $x/a^t$ is
\begin{align}
\label{eq:cardl}
(tx)^{1/(t+1)}+O(1).
\end{align}

Adding \eqref{eq:cards} and \eqref{eq:cardl}, we have
$$\left|S_{\fl{\frac{x}{n^t}},x}\right|= t^{-\frac{t}{t+1}}x^{\frac1{t+1}}+(tx)^{1/(t+1)}+O(1)=x^{\frac1{t+1}}\(t^{-\frac{t}{t+1}}+t^{\frac1{t+1}}\)+O(1),$$
concluding the proof.
\end{proof}

The following lemma will also be needed. The proof is rather long, so it will be postponed to Section~\ref{se:o(x)}.
\begin{lemma}
    \label{lem:o(x)}
    Let $t \geqslant 1$. For all $x \geqslant 2 t^t$ and $\epsilon >0$, we have
 \begin{align}
 \label{psi sum}
    &\sum_{t^{-\frac{t}{t+1}}x^{\frac1{t+1}}<p\leq x^{\frac{155}{154(t+1)}+\epsilon}}\left(\psi\left(\sqrt[t]{\frac{x}{p+1}}\right)-\psi\left(\sqrt[t]{\frac{x}{p}}\right)\right)\notag\\ &\qquad\qquad\qquad\ll x^{\frac{5890}{5929(t+1)}-\frac{43}{11858t}+\epsilon}
    \approx x^{\frac{0.9934}{t+1}-\frac{0.0036}{t}+\epsilon}.
\end{align}
\end{lemma}

\section{Proof of Theorem \ref{thm:x/n^t}}

\subsection{Assumption regarding size of $x$}

\noindent
In this subsection we give the lower bound for $x$ in Theorem \ref{thm:x/n^t}, and outline the calculations behind the maximum calculations.
In what follows,
let $x \geqslant e$ and $t > 1$ be fixed, and we always assume $x$ to be large enough to ensure the inequality
\begin{equation}
   x \geqslant \max \Bigl( (368t)^{385(t+1)(N+2)}, (\log x)^{154t(t+1)(N+2)},t^{2t^2} , t^t e^{818(t+1)(N+2)^2} \Bigr) \label{eq:hyp_x}
\end{equation}
for all $N \in \Z_{\geqslant 0}$.
Set $a := \left( x t\right)^{\frac{1}{t+1}} $ so that $x a^{-t} = a t^{-1} = \left( x t^{-t}\right)^{\frac{1}{t+1}}$. Also let $\varepsilon$ be any real number such that $0 < \varepsilon < \frac{1}{500(t+1)}$. Finally, for all $c > 0$, let $\omega_c(x) := e^{-c (\log x)^{3/5} (\log \log x)^{-1/5}}$\footnote{The standard notation for this function is $\delta_c(x)$ but $\delta$ is already used in this paper.} be the usual number-theoretic remainder function. Notice that:
\begin{enumerate}
   \item[\scriptsize $\triangleright$] $x \geqslant (368t)^{385(t+1)(N+2)}$ along with $0 < \varepsilon < \frac{1}{500(t+1)}$ ensures that the error term in \eqref{eq:psi sum} is absorbed by the error term in \eqref{eq:sigma_1}.
   \item[\scriptsize $\triangleright$] $x \geqslant (\log x)^{154t(t+1)(N+2)}$ ensures that the error term in \eqref{eq:error_term} is absorbed by the error term in \eqref{eq:sigma_1}.
   \item[\scriptsize $\triangleright$] $x \geqslant t^{2t^2}$ ensures that $\frac{t \log t}{\log x} \leqslant \frac{t^2}{\log x} \leqslant \frac{t^2 \log t}{\log x} \leqslant \frac{1}{2}$, used throughout.
   \item[\scriptsize $\triangleright$] $x \geqslant t^te^{818(t+1)(N+2)^2}$ ensures that $\omega_c \left( xa^{-t} \right) \leqslant \left( \log xa^{-t}\right)^{-N-2}$, where it is known \cite{ford02} that we may take $c \approx \np{0.2098}$. See Lemma~\ref{lem:sum_p_beta_inv}.
\end{enumerate}

\subsection{Splitting $\pi\(S_{\fl{\frac{x}{n^t}},x}\)$ into subsums}\label{SS:Preliminaries}

\noindent
The proof estimates
$$\pi \left( S_{\left\lfloor\frac{x}{n^t}\right\rfloor,x}\right) = \sum_{p \leqslant x} \left( \left \lfloor \sqrt[t]{\frac{x}{p}} \right \rfloor - \left \lfloor \sqrt[t]{\frac{x}{p+1}} \right \rfloor \right).$$
This sum is split into three subsums as follows:
\begin{align}
   \pi \left( S_{\left\lfloor\frac{x}{n^t}\right\rfloor,x}\right) &= \left( \sum_{p \leqslant xa^{-t}} + \sum_{xa^{-t} < p \leqslant x^{\frac{155}{154(t+1)}}} + \sum_{p > x^{\frac{155}{154(t+1)}}} \right) \left( \left \lfloor \sqrt[t]{\frac{x}{p}} \right \rfloor - \left \lfloor \sqrt[t]{\frac{x}{p+1}} \right \rfloor \right) \notag \\
   &= S_1(x,t) + S_2(x,t) + S_3(x,t). \label{eq:s}
\end{align}

\noindent
As often in this situation, the first subsum yields the main term plus an acceptable error term, and the third one can be estimated trivially. The difficult part comes from the second sum, for which a new estimate on a sum over primes is needed. This latter sum is classically derived using the very useful Vaughan's identity, for which we provide here a slightly different version. See Lemma~\ref{lem:Vaughan_bis}.
\label{lem:Vaughan_bis}

\subsection{Basic tools}
\label{SS:Basic tools}
\noindent
In this subsection we give the Lemmas used for the estimation of $\sum_{t,1}(x)$, used in turn to estimate $S_2(x,t)$.

We will make frequent use of the following simple bound.

\begin{lemma}
\label{le:bound_log}
Let $j \in \Z_{\geqslant 0}$, $t > 1$ and $x \geqslant e$ such that $x \geqslant t^{2t}$. Then
$$\frac{1}{\left( \log xa^{-t} \right)^j} \leqslant \left( \frac{2(t+1)}{\log x}\right)^j.$$
\end{lemma}

\begin{proof}
The left-hand side is equal to $\left( \frac{t+1}{\log x}\right)^j \times \left( 1 - \frac{t \log t}{\log x}\right)^{-j}$, and the result follows using the inequality $(1-X)^{-1} \leqslant 2$ whenever $X \leqslant \frac{1}{2}$.
\end{proof}

\noindent A more accurate estimate can be proved if we used the series expansion of $(1-X)^{-K}$.

\begin{lemma}
\label{le:expansion_log}
Let $j \in \Z_{\geqslant 0}$, $t > 1$ and $x \geqslant e$ such that $x \geqslant t^{2t}$. Then, for all $N,j \in \Z_{\geqslant 0}$
\begin{multline*}
  \frac{1}{(\log xa^{-t})^{j+1}} = \left( \frac{t+1}{\log x} \right)^{j+1} \sum_{k=0}^N {j+k \choose k} \left( \frac{t \log t}{\log x} \right)^k \\
  + O_{j,N} \left( \frac{t^{j+N+2}(\log t)^{N+1}}{(\log x)^{j+N+2}}\right).
\end{multline*}
\end{lemma}

\begin{proof}
As above, the \textsc{lhs} is written as $\left( \frac{t+1}{\log x}\right)^{j+1} \times \left( 1 - \frac{t \log t}{\log x}\right)^{-j-1}$, and we conclude the proof using the series expansion $(1-X)^{-K} = \sum_{k=0}^\infty {K+k-1 \choose k} X^k$ for $|X|<1$.
\end{proof}

\begin{lemma}
\label{lem:sum_p_beta_inv_0}
Let $\beta > 1$ be fixed. Then, for all $x \geqslant 5$
$$\sum_{p > x} \frac{1}{p^\beta} < \frac{2 \beta}{\beta - 1} \, \frac{1}{x^{\beta-1} \log x}.$$
\end{lemma}

\begin{proof}
Follows at once from \cite[Lemma~3.2]{ram16}.
\end{proof}

\begin{lemma}
\label{lem:sum_p_beta_inv}
Let $1 < \beta \leqslant \beta_0$ be fixed. Then, for all $N \in \Z_{\geqslant 0}$ and all $x > \max \left( e^{\frac{1}{\beta-1}}, e^{818(N+2)^2} \right)$,
\begin{multline*}
  \sum_{p > x} \frac{1}{p^\beta} = \frac{1}{\beta-1} \frac{1}{x^{\beta-1} \log x} \sum_{j=0}^N \frac{(-1)^j j!}{(\beta-1)^j (\log x)^j} \\
  + O_N \left( \frac{\beta_0}{x^{\beta-1}((\beta - 1)\log x)^{N+2}} \right).
\end{multline*}
The error term depends only on $N$.
\end{lemma}

\begin{proof}
By partial summation and the prime number theorem, there exits $c >0$ such that
\begin{align*}
   \sum_{p > x} \frac{1}{p^\beta} &= - \frac{\pi(x)}{x^\beta} + \beta \int_x^\infty t^{-\beta-1} \pi(t) \, \textrm{d}t \\
   &= - \frac{\Li(x)}{x^\beta} + O \left( x^{1-\beta} \omega_c(x)\right) + \beta \int_x^\infty t^{-\beta-1} \left\{ \Li(t) + O \left( t \omega_c(t) \right) \right\} \, \textrm{d}t \\
   &= - \frac{\Li(x)}{x^\beta} + \beta \int_x^\infty t^{-\beta-1} \Li(t) \, \textrm{d}t + O \left( \frac{\beta_0 x^{1-\beta} \omega_c(x)}{\beta-1}\right) \\
   &= \int_x^\infty \frac{\textrm{d}t}{t^{\beta}\log t}  + O \left( \frac{\beta_0 x^{1-\beta} \omega_c(x)}{\beta-1}\right) \\
   &= \frac{1}{\beta-1} \frac{1}{x^{\beta-1} \log x} \sum_{j=0}^N \frac{(-1)^j j!}{(\beta-1)^j (\log x)^j} + O_N \left( \frac{\beta_0}{x^{\beta-1}((\beta - 1)\log x)^{N+2}} \right)
\end{align*}
by successive integrations by part.
\end{proof}

\subsection{The sum $S_1(x,t)$}
\label{SS:S1}
\noindent
By the prime number theorem, there exists $c > 0$ such that
$$S_1(x,t) = \pi \left( xa^{-t} \right) = \Li \left( xa^{-t} \right) + O \left( xa^{-t} \omega_c \left( xa^{-t} \right) \right) $$
so that, using the asymptotic expansion of $\Li$, we derive for all $N \in \Z_{\geqslant 0}$,
$$S_1(x,t) = \frac{xa^{-t}}{\log xa^{-t}} \sum_{j=0}^N \frac{j!}{(\log xa^{-t})^j} + O \left( \frac{xa^{-t}}{(\log xa^{-t})^{N+2}}\right).$$
Now using Lemma~\ref{le:bound_log} for the error term and Lemma~\ref{le:expansion_log} for the main term yields
\begin{multline}
   S_1(x,t) = \left( \frac{x}{t^t}\right)^{\frac{1}{t+1}} \frac{t+1}{\log x} \sum_{\substack{j,k=0 \\ j+k \leqslant N}}^N j! {j+k \choose k} \frac{t^{k}(t+1)^j (\log t)^k}{(\log x)^{j+k}} \\
   + O_N \left( \left( \frac{x}{t^t}\right)^{\frac{1}{t+1}} \frac{t}{\log x} \left( \frac{t \log t}{\log x} \right)^{N+1} \right) \label{eq:s_1}
\end{multline}
provided that $x \geqslant t^{2t}$.

\subsection{The sum $S_2(x,t)$}

\noindent
The sum $S_2(x,t)$ is split into 2 subsums. The subsum $\Sigma_{t,1}(x)$ is estimated in subsubsection \ref{SSS:S1}. The estimation of subsum $\Sigma_{t,2}(x)$ is the result of Lemma \ref{lem:o(x)} and is proven in Section \ref{se:o(x)}.

We first write
\begin{align*}
   S_2(x,t) &= \sum_{xa^{-t} < p \leqslant x^{\frac{155}{154(t+1)}}} \left( \sqrt[t]{\frac{x}{p}} - \sqrt[t]{\frac{x}{p+1}} \right) \\
   & \hspace*{2cm} + \sum_{xa^{-t} < p \leqslant x^{\frac{155}{154(t+1)}}} \left( \psi \left(  \sqrt[t]{\frac{x}{p+1}} \right) - \psi \left( \sqrt[t]{\frac{x}{p}} \right) \right) \\
   &:= \Sigma_{t,1}(x) + \Sigma_{t,2}(x).
\end{align*}

\subsubsection{The sum $\Sigma_{t,1}(x)$}
\label{SSS:S1}
\noindent
We have
\begin{align*}
   \Sigma_{t,1}(x) &= \frac{x^{1/t}}{t} \sum_{xa^{-t} < p \leqslant x^{\frac{155}{154(t+1)}}} \frac{1}{p^{1+1/t}} + O \left( \frac{x^{1/t}}{t} \sum_{xa^{-t} < p \leqslant x^{\frac{1}{t+1}+\epsilon}} \frac{1}{p^{2+1/t}}\right) \\
   &= \frac{x^{1/t}}{t} \sum_{p > xa^{-t}} \frac{1}{p^{1+1/t}} + O \left( \frac{x^{1/t}}{t} \sum_{p > x^{\frac{155}{154(t+1)}}} \frac{1}{p^{1+1/t}} + \frac{x^{1/t}}{t} \sum_{p > xa^{-t}} \frac{1}{p^{2+1/t}}\right).
\end{align*}
For $x \geqslant \max \left( t^{2t}, 5(5t)^t \right)$, using Lemma~\ref{lem:sum_p_beta_inv_0} with $\beta = 1+\frac{1}{t}$ in the $1$st sum of the error term and $\beta = 2 + \frac{1}{t}$ in the $2$nd sum of the error term, and also using Lemma~\ref{le:bound_log}, the error term does not exceed
\begin{align}
   & < \frac{8t x^{\frac{1}{t+1}-\frac{1}{154t(t+1)}}}{\log x} + \frac{4}{t^{t} \log \left( xa^{-t} \right)} \notag \\
   & < \frac{8t}{\log x} \left( x^{\frac{1}{t+1}-\frac{1}{154t(t+1)}} + 2 t^{-t} \right) \notag \\
   & \leqslant \frac{24t x^{\frac{1}{t+1}-\frac{1}{154t(t+1)}}}{\log x}  < 24 x^{\frac{1}{t+1}-\frac{1}{154t(t+1)}}. \label{eq:error_term}
\end{align}
For $x \geqslant \max \left( t^{2t^2}, t^t e^{818(t+1)(N+2)^2} \right) $, using Lemma~\ref{lem:sum_p_beta_inv} with $\beta_0=2$ and $\beta = 1+\frac{1}{t}$ , the main term is
\begin{align}
   &= \frac{\left( x t\right)^{\frac{1}{t+1}}}{\log \left( x a^{-t}\right) } \sum_{j=0}^N \frac{(-t)^j j!}{(\log xa^{-t})^j} + O \left( \frac{\left( x t\right)^{\frac{1}{t+1}}t^{N+1}}{\left( \log x a^{-t} \right)^{N+2}} \right) \notag \\
   &= \left( \frac{x}{t^t}\right)^{\frac{1}{t+1}} \frac{t}{\log \left( x a^{-t}\right) } \sum_{j=0}^N \frac{(-t)^j j!}{(\log xa^{-t})^j} + O_N \left( \left( \frac{x}{t^t}\right)^{\frac{1}{t+1}} \frac{t^{N+2}}{\left( \log xa^{-t} \right)^{N+2}} \right) \notag \\
   &= \left( \frac{x}{t^t}\right)^{\frac{1}{t+1}} \frac{t(t+1)}{\log x} \sum_{\substack{j,k=0 \\ j+k \leqslant N}}^N (-1)^j j! {j+k \choose k} \frac{t^{j+k}(t+1)^j (\log t)^k}{(\log x)^{j+k}} \notag \\
   & \hspace*{2cm} + O_N \left( \left( \frac{x}{t^t}\right)^{\frac{1}{t+1}} \frac{t^2}{\log x} \left( \frac{t^2 \log t}{\log x} \right)^{N+1} \right) \label{eq:main_term}
\end{align}
where we used Lemma~\ref{le:expansion_log} and Lemma~\ref{le:bound_log} in the last line. Note that the condition \eqref{eq:hyp_x} implies that the error term in \eqref{eq:error_term} is absorbed by the error term in \eqref{eq:main_term}. Hence, for all $N \in \Z_{\geqslant 0}$ and all $x$ satisfying \eqref{eq:hyp_x}, we obtain
\begin{multline}
   \Sigma_{t,1}(x) = \left( \frac{x}{t^t}\right)^{\frac{1}{t+1}} \frac{t(t+1)}{\log x} \sum_{\substack{j,k=0 \\ j+k \leqslant N}}^N (-1)^j j! {j+k \choose k} \frac{t^{j+k}(t+1)^j (\log t)^k}{(\log x)^{j+k}} \\
   + O_N \left( \left( \frac{x}{t^t}\right)^{\frac{1}{t+1}} \frac{t^2}{\log x} \left( \frac{t^2 \log t}{\log x} \right)^{N+1} \right). \label{eq:sigma_1}
\end{multline}

\subsubsection{The sum $\Sigma_{t,2}(x)$}

\noindent
Using Lemma~\ref{lem:o(x)} we derive at once

\begin{equation}
    \Sigma_{t,2}(x) \ll x^{\frac{\np{5890}}{\np{5929}(t+1)} - \frac{43}{\np{11858}t}  + \varepsilon} \approx x^{\frac{\np{0.9934}}{t+1} - \frac{\np{0.0036}}{t} + \varepsilon}\label{eq:psi sum}
\end{equation}
for all $x \geqslant 2t^{t}$ and all $\varepsilon >0$ small. Note that if $x$ satisfies \eqref{eq:hyp_x} and $0 < \varepsilon < \frac{1}{500(t+1)}$, then the bound \eqref{eq:psi sum} is absorbed by the error term in \eqref{eq:sigma_1} above.

\subsubsection{The sum $S_2(x,t)$}

\noindent
By \eqref{eq:sigma_1} and \eqref{eq:psi sum} we immediately derive the next estimate. Assuming $x,t > 1$ satisfying \eqref{eq:hyp_x}, we have
\begin{multline}
   S_2(x,t) = \left( \frac{x}{t^t}\right)^{\frac{1}{t+1}} \frac{t(t+1)}{\log x} \sum_{\substack{j,k=0 \\ j+k \leqslant N}}^N (-1)^j j! {j+k \choose k} \frac{t^{j+k}(t+1)^j (\log t)^k}{(\log x)^{j+k}} \\
   + O \left( \left( \frac{x}{t^t}\right)^{\frac{1}{t+1}} \frac{t^2}{\log x} \left( \frac{t^2 \log t}{\log x} \right)^{N+1} \right). \label{eq:s_2}
\end{multline}

\subsection{The sum $S_3(x,t)$}
\label{SS:S3}

\noindent
If $p\in S_{\fl{x/n^t},x}$ satisfies $p>x^{\frac{155}{154(t+1)}}$, then we must have $n<x^{\frac{1}{t+1}-\frac{1}{154t(t+1)}}$ so that
\begin{equation}
   S_3(x,t) \ll x^{\frac{1}{t+1}-\frac{1}{154t(t+1)}}, \label{eq:s_3}
\end{equation}
which is absorbed in the error term of \eqref{eq:s_2} if $x,t$ satisfy \eqref{eq:hyp_x}.

\subsection{Conclusion}

\noindent
Substituting \eqref{eq:s_1}, \eqref{eq:s_2} and \eqref{eq:s_3} into \eqref{eq:s} yields Theorem~\ref{thm:x/n^t}.
\qed

\section{Proof of Lemma~\ref{lem:o(x)}}
\label{se:o(x)}

This section estimates the sum $\sum_{t,2}(x)$. In subsection \ref{SS:tech lem} we combine the Dirichlet hyperbola principle with Vaughan's identity for the Möbius function to equate a  von Mangoldt function twisted exponential sum  with 4 triple exponential sums and a double exponential sum. These five sums are estimated in Subsection \ref{SS:exp sum prime}. Finally, in subsection \ref{SS:proof lem 3}, we use the previous lemma with specified exponential to estimate $\sum_{t,2}(x)$.

\subsection{Technical lemmas}
\label{SS:tech lem}

The first tool is a direct application of the celebrated Dirichlet hyperbola principle. See \cite[Corollary~3.4]{Bor2}.

\begin{lemma}
\label{lem:Dirichlet_exponential_principle}
Let $f,g: \Z_{\geqslant 1} \to \C$ be two arithmetic functions and $F : \left[ 1,\infty \right) \to \left[ 0,\infty \right)$ be any map. For all $R<R_1 \in \Z_{\geqslant 1}$ and $1 \leqslant U \leqslant R$
\begin{multline*}
   \sum_{R < n \leqslant R_1} \left( f \star g \right)(n)  \, e \left( F(n) \right) = \sum_{n \leqslant \frac{UR_1}{R}} f(n) \sum_{\frac{R}{n} < m \leqslant \frac{R_1}{n}} g(m) \, e \left( F(mn) \right) \\
   + \sum_{n \leqslant \frac{R}{U}} g(n) \sum_{\frac{R}{n} < m \leqslant \frac{R_1}{n}} f(m)  \, e \left( F(mn) \right) - \sum_{U < n \leqslant \frac{UR_1}{R}} f(n) \sum_{\frac{R}{n} < m \leqslant \frac{R}{U}} g(m)  \, e \left( F(mn) \right).
\end{multline*}
\end{lemma}

Our second tool is an application of the no less famous Vaughan's identity for the M\"{o}bius function. See \cite[(11)]{monva81}.

\begin{lemma}
\label{lem:mu}
Let $f: \Z_{\geqslant 1} \to \C$ be an arithmetic function and $1 \leqslant R < R_1 \leqslant 2R$ be positive integers. For all $1 \leqslant V \leqslant R^{1/2}$
\begin{multline*}
 \sum_{R < n \leqslant R_1} \mu(n) f(n) = -\sum_{n \leqslant V^2} a_n \sum_{\frac{R}{n} < m \leqslant \frac{R_1}{n}} f(mn) \\
 - \sum_{V < n \leqslant \frac{R_1}{V}} b_n \sum_{\max \left( V, \frac{R}{n} \right) < m \leqslant \frac{R_1}{n}} \mu(m) f(mn),
\end{multline*}
where $a_n \ll n^\epsilon$ and $b_n \ll n^\epsilon$.
\end{lemma}

Combining these two lemmas yield the next identity which may be seen as an alternate Vaughan's decomposition for the von Mangoldt function, with triple exponential sums instead of double exponential sums.

\begin{lemma}
\label{lem:Vaughan_bis}
Let $F : \R_+ \to \R$ be any map and $1 \leqslant R < R_1 \leqslant 2R$ be positive integers. For all $1 \leqslant U \leqslant R$ and $1 \leqslant V \leqslant U^{1/2}$ we have
\begin{align*}
   \sum_{R < n \leqslant R_1} \Lambda(n) \, e \left( F(n)\right) & = \sum_{n \leqslant 2U} \mu(n) \sum_{\frac{R}{n} < m \leqslant \frac{R_1}{n}} e \left( F(mn)\right) \log m \\
   & \hspace*{0.5cm}- \sum_{n \leqslant \frac{R}{U}} \log n \sum_{m \leqslant V} a_m \sum_{\frac{R}{mn} < h \leqslant \frac{R_1}{mn}} e \left( F(mnh)\right) \log h \\
   & \hspace*{1cm} - \sum_{n \leqslant \frac{R}{U}} \log n \sum_{V < m \leqslant V^2} a_m \sum_{\frac{R}{mn} < h \leqslant \frac{R_1}{mn}} e \left( F(mnh)\right) \log h \\
   & \hspace*{1.5cm} + \sum_{n \leqslant \frac{R}{U}} \log n \sum_{V < m \leqslant \frac{R_1}{Vn}} b_m \sum_{\max \left( \frac{R}{mn},V \right)  < h \leqslant \frac{R_1}{mn}} \mu(h) \, e \left( F(mnh)\right) \\
   & \hspace*{2cm} - \sum_{U < n \leqslant 2U} \mu(n) \sum_{\frac{R}{n} < m \leqslant \frac{R}{U}} e \left( F(mn)\right) \log m,
\end{align*}
with $a_m \ll m^\epsilon$ and $b_m \ll m^\epsilon$.
\end{lemma}

\begin{proof}
Lemma~\ref{lem:Dirichlet_exponential_principle} with $f=\mu$ and $g=\log$ yields
\begin{multline*}
   \sum_{R < n \leqslant R_1} \Lambda(n) \, e \left( F(n) \right) = \sum_{n \leqslant 2U} \mu(n) \sum_{\frac{R}{n} < m \leqslant \frac{R_1}{n}} e \left( F(mn) \right) \log m \\
   + \sum_{n \leqslant \frac{R}{U}} \log n \sum_{\frac{R}{n} < m \leqslant \frac{R_1}{n}} \mu(m)  \, e \left( F(mn) \right) - \sum_{U < n \leqslant 2U} \mu(n) \sum_{\frac{R}{n} < m \leqslant \frac{R}{U}} e \left( F(mn) \right) \log m,
\end{multline*}
and we treat the inner sum of the $2$nd sum with Lemma~\ref{lem:mu}. The assumption $V \leqslant (R/n)^{1/2}$ for all $1 \leqslant n \leqslant R/U$ is then ensured via $(R/n)^{1/2} \geqslant U^{1/2} \geqslant V$. This $2$nd sum is then equal to
\begin{align*}
   &= \sum_{n \leqslant \frac{R}{U}} \log n \left\lbrace -\sum_{m \leqslant V} a_m \sum_{\frac{R}{mn} < h \leqslant \frac{R_1}{mn}} e \left( F(mnh)\right) \log h \right. \\
   & \qquad \left. - \sum_{V < m \leqslant V^2} a_m \sum_{\frac{R}{mn} < h \leqslant \frac{R_1}{mn}} e \left( F(mnh)\right) \log h \right. \\
   & \qquad \left. + \sum_{V < m \leqslant \frac{R_1}{Vn}} b_m \sum_{\max \left( \frac{R}{mn},V \right)  < h \leqslant \frac{R_1}{mn}} \mu(h) \, e \left( F(mnh)\right) \right\rbrace,
\end{align*}
as required.
\end{proof}

Finally, we will make use of the following estimate which is currently one of the best for estimating triple exponential sums of type II. See \cite{sarrob04}.

\begin{lemma}
\label{lem:sarrob04}
Let $X>0$, $H,M,N \in \Z_{\geqslant 1}$, $\left( a_h \right)$, $\left( b_{m,n}\right) \in \C$ such that $\left| a_h\right| ,\left| b_{m,n}\right| \leqslant 1$, $\alpha ,\beta ,\gamma \in \R$ such that $\alpha \left( \alpha -1\right) \beta \gamma \neq 0$ and $\epsilon >0$. Then
\begin{align*}
   & \left( HMN\right)^{-\epsilon }\sum_{H < h \leqslant 2H} a_h \sum_{M < m \leqslant 2M} \sum_{N < n \leqslant 2N} b_{m,n} \, e\left( X\left( \frac hH\right)^{\alpha }\left( \frac mM\right)^{\beta }\left( \frac nN\right)^{\gamma } \right) \\
   & \ll \left( X H^2 M^3 N^3 \right)^{1/4} + H \left( MN\right)^{3/4} + H^{1/2} M N + X^{-1/2} H M N.
\end{align*}
\end{lemma}

\subsection{Exponential sums over primes}
\label{SS:exp sum prime}
The aim of this section is the proof of the next estimate.

\begin{lemma}
\label{lem:Lambda_exp_bord}
Let $\alpha, c_0, \epsilon > 0 \in \R_{> 0}$. Uniformly for $x \geqslant e$, $1 \leqslant R < R_1 \leqslant 2R$ such that
\begin{equation}
   R \leqslant c_0 x^{\frac{6}{6 \alpha+1}} \label{eq:cond_lemma}
\end{equation}
we have
\begin{multline*}
   x^{-\epsilon} \sum_{R < n \leqslant R_1} \Lambda(n) e \left( \frac{x}{n^\alpha} \right) \ll T^{\frac{5k-4 \ell+4}{16(k-\ell)+17}} R^{\frac{11(k - \ell)+12}{16(k-\ell)+17}} \\
   + T^{\frac{k}{16(k-\ell)+17}} R^{\frac{15(k - \ell)+16}{16(k-\ell)+17}} + T^{1/4} R^{2/3} + T^k R^{\ell-k} + R^{11/12},
\end{multline*}
where $(k,\ell)$ is an exponent pair and $T:= xR^{-\alpha}$.
\end{lemma}

\begin{proof}
Let $S$ be the sum of the left-hand side, $1 \leqslant U \leqslant R$ and $1 \leqslant V \leqslant U^{1/2}$. We use Lemma~\ref{lem:Vaughan_bis} with $F(n) = xn^{- \alpha}$, and for which we denote the five sums by $S_1,\dotsc,S_5$. Lemma~\ref{lem:sarrob04} will be applied with $\left( \alpha, \beta, \gamma \right) = \left( - \alpha, - \alpha, - \alpha \right)$ and $X := x(MNH)^{- \alpha}$. We also set $\ell_n := R^{-\epsilon} \log n$, $\alpha_m := R^{-\epsilon} a_m$ and $\beta_m := R^{-\epsilon} b_m$. Finally, recall that any multiplicative constraint $R < hmn \leqslant R_1$ appearing in the sums below may be omited at a cost of a factor $(\log R)^2$.
Now using the exponent pair $(k,\ell)$, we first get
\begin{align*}
   S_1 & \ll \sum_{n \leqslant 2U} \ \underset{\frac{R}{n} \leqslant M \leqslant \frac{R_1}{n}}{\max} \ \left| \sum_{\frac{R}{n} < m \leqslant M} e \left( \frac{x}{(mn)^\alpha} \right) \right| \ \log R \\
   & \ll \sum_{n \leqslant 2U} \left\lbrace x^k R^{\ell - k(\alpha+1)} n^{k-\ell} + \frac{R^{1+\alpha}}{nx} \right\rbrace \log R \\
   & \ll x^k R^{\ell - k(\alpha+1)} U^{1-\ell+k} \log R + R^{1+\alpha} x^{-1} (\log R)^2.
\end{align*}
Then, using Lemma~\ref{lem:sarrob04} with $(H,M,N) \leftrightarrow (M,N,H)$, we derive
\begin{align*}
   R^{- \epsilon} S_3 & \ll \max_{N \leqslant\frac{R}{U}} \ \max_{V < M \leqslant V^2} \ \max_{\substack{H \geqslant 1 \\ HMN \asymp R}} \left| \underset{R < hmn \leqslant R_1}{\sum_{N < n \leqslant 2N} \ell_n \sum_{M < m \leqslant 2M} \alpha_m \sum_{H < h \leqslant 2H} \alpha_h \, e \left( \frac{x}{(hmn)^\alpha}\right)}  \right| \\
   & \ll \max_{N \leqslant\frac{R}{U}} \ \max_{V < M \leqslant V^2} \ \max_{\substack{H \geqslant 1 \\ HMN \asymp R}} \left\lbrace \left(x M^{2-\alpha} (NH)^{3-\alpha} \right)^{1/4} + M(NH)^{3/4} \right. \\
   & \qquad \left. + \, M^{1/2} NH + x^{-1/2} (MNH)^{1+\alpha/2} \right\rbrace \\
   & \ll \max_{N \leqslant\frac{R}{U}} \ \max_{V < M \leqslant V^2} \left\lbrace \left( xM^{-1}\right)^{1/4} R^{\frac{3-\alpha}{4}} + (R^3M)^{1/4} + RM^{-1/2} + x^{-1/2} R^{1+\alpha/2} \right\rbrace \\
   & \ll \left( xV^{-1}\right)^{1/4} R^{\frac{3-\alpha}{4}} + (R^3V^2)^{1/4} + RV^{-1/2} + x^{-1/2} R^{1+\alpha/2}.
\end{align*}
Similarly,
\begin{align*}
   R^{- \epsilon} S_4 & \ll \max_{N \leqslant\frac{R}{U}} \ \max_{V < M \leqslant \frac{R_1}{V}} \ \max_{\substack{H \geqslant 1 \\ HMN \asymp R}} \left| \underset{R < hmn \leqslant R_1}{\sum_{N < n \leqslant 2N} \ell_n \sum_{M < m \leqslant 2M} \beta_m \sum_{H < h \leqslant 2H} \alpha_h \, e \left( \frac{x}{(hmn)^\alpha}\right)}  \right| \\
   & \ll \max_{N \leqslant\frac{R}{U}} \ \max_{V < M \leqslant \frac{R_1}{V}} \ \max_{\substack{H \geqslant 1 \\ HMN \asymp R}} \left\lbrace \left(x M^{2-\alpha} (NH)^{3-\alpha} \right)^{1/4} + M(NH)^{3/4} \right. \\
   & \qquad \left. + \, M^{1/2} NH + x^{-1/2} (MNH)^{1+\alpha/2} \right\rbrace \\
   & \ll \max_{N \leqslant\frac{R}{U}} \ \max_{V < M \leqslant \frac{2R}{V}} \left\lbrace \left( xM^{-1}\right)^{1/4} R^{\frac{3-\alpha}{4}} + (R^3M)^{1/4} + RM^{-1/2} + x^{-1/2} R^{1+\alpha/2} \right\rbrace \\
   & \ll \left( xV^{-1}\right)^{1/4} R^{\frac{3-\alpha}{4}} + RV^{-1/4} + x^{-1/2} R^{1+\alpha/2}.
\end{align*}
Next, we split $S_2$ into two subsums
   $$S_2 = \left( \sum_{n \leqslant U} + \sum_{U < n \leqslant \frac{R}{U}} \right) \log n \sum_{m \leqslant V} a_m \sum_{\frac{R}{mn} < h \leqslant \frac{R_1}{mn}} e \left( F(mnh)\right) \log h := S_{21} + S_{22},$$
and, using Lemma~\ref{lem:sarrob04} with $(H,M,N) \leftrightarrow (N,M,H)$, we get
\begin{align*}
   R^{-\epsilon} S_{22} & \ll \max_{U < N \leqslant \frac{R}{U}} \ \max_{M \leqslant V} \, \max_{\substack{H \geqslant 1 \\ HMN \asymp R}} \left| \underset{R < hmn \leqslant R_1}{\sum_{N < n \leqslant 2N} \ell_n \sum_{M < m \leqslant 2M} \alpha_m \sum_{H < h \leqslant 2H} \ell_h \, e \left( F(mnh)\right)}  \right| \\
   & \ll \max_{U < N \leqslant \frac{R}{U}} \ \max_{M \leqslant V} \, \max_{\substack{H \geqslant 1 \\ HMN \asymp R}} \left\lbrace \left(x N^{2-\alpha} (MH)^{3-\alpha} \right)^{1/4} + N(MH)^{3/4} \right. \\
   & \qquad \left. + \, N^{1/2} MH + x^{-1/2} (MNH)^{1+\alpha/2} \right\rbrace \\
   & \ll \max_{U < N \leqslant \frac{R}{U}} \ \max_{M \leqslant V}\left\lbrace \left( xN^{-1}\right)^{1/4} R^{\frac{3-\alpha}{4}} + (R^3N)^{1/4} + RN^{-1/2} + x^{-1/2} R^{1+\alpha/2} \right\rbrace \\
   & \ll  \left( xU^{-1}\right)^{1/4} R^{\frac{3-\alpha}{4}} + RU^{-1/4} + RU^{-1/2} + x^{-1/2} R^{1+\alpha/2}\footnotemark,
\end{align*}
and the exponent pair $(k,\ell)$ yields
\begin{align*}
   R^{-\epsilon} S_{21} & \ll \sum_{n \leqslant U} \ \sum_{m \leqslant V} \ \underset{\frac{R}{mn} \leqslant H \leqslant \frac{R_1}{mn}}{\max} \ \left| \sum_{\frac{R}{mn} < h \leqslant H} e \left( F(mnh)\right) \right| \\
   & \ll \sum_{n \leqslant U} \ \sum_{m \leqslant V} \frac{1}{(mn)^{\ell-k}} \left\lbrace x^k R^{\ell - k(\alpha+1)} + \frac{R^{1+\alpha}}{mnx} \right\rbrace \\
   & \ll x^k R^{\ell - k(\alpha+1)} (UV)^{1-\ell+k}  + R^{1+\alpha} x^{-1}.
\end{align*}
Finally, using $1 \leqslant V \leqslant U^{1/2}$, the full sum $S$ does not exceed
\begin{align*}
   R^{-\epsilon} S & \ll x^k R^{\ell - k(\alpha+1)} (UV)^{1-\ell+k} + \left( x^{1/4} R^{\frac{3-\alpha}{4}} + R \right) V^{-1/4} \\
   & \qquad  + (R^3V^2)^{1/4} +  x^{-1/2} R^{1+\alpha/2} + R^{1+\alpha} x^{-1} \\
   & \ll x^k R^{\ell - k(\alpha+1)} (UV)^{1-\ell+k} + x^{1/4} R^{\frac{3-\alpha}{4}} V^{-1/4} \\
   & \qquad + RV^{-1/4} + (R^3V^2)^{1/4} + x^{-1/2} R^{1+\alpha/2},
\end{align*}
where we also used the fact that the condition $R \leqslant c_0 x^{\frac{6}{6 \alpha + 1}}$ implies that $x^{-1/2} R^{1+\alpha/2} \geqslant c_0^{- \alpha/2} R^{1+\alpha} x^{-1}$. Now choose $V = U^{1/3}$, so that $V \leqslant R^{1/3}$. This inequality implying that $(R^3V^2)^{1/4} \leqslant RV^{-1/4}$, and therefore
$$R^{-\epsilon} S \ll x^k R^{\ell - k(\alpha+1)} U^{4(1-\ell+k)/3} + x^{1/4} R^{\frac{3-\alpha}{4}} U^{-1/12} + RU^{-1/12} + x^{-1/2} R^{1+\alpha/2}.$$
Optimizing over $U \in \left[ 1,R \right]$ yields
\begin{align*}
   x^{-\epsilon} \sum_{R < n \leqslant R_1} \Lambda(n) e \left( \frac{x}{n^\alpha} \right) & \ll T^{\frac{5k-4 \ell+4}{16(k-\ell)+17}} R^{\frac{11(k - \ell)+12}{16(k-\ell)+17}} + T^{\frac{k}{16(k-\ell)+17}} R^{\frac{15(k - \ell)+16}{16(k-\ell)+17}} \\
    & \qquad + T^{1/4} R^{2/3} + T^k R^{\ell-k} + R^{11/12} + x^{-1/2} R^{1+\alpha/2},
\end{align*}
and we conclude the proof noticing that the term $x^{-1/2} R^{1+\alpha/2}$ is absorbed by $R^{11/12}$ since $R \leqslant c_0 x^{\frac{6}{6 \alpha + 1}}$.
\end{proof}

\subsection{Proof of Lemma~\ref{lem:o(x)}}
\label{SS:proof lem 3}
We will use Lemma~\ref{lem:Lambda_exp_bord} with van der Corput's exponent pair $(k,\ell) = \left( \frac{1}{2} , \frac{1}{2} \right)$, which yields
\begin{equation}
   x^{-\epsilon} \sum_{N < n \leqslant N_1} \Lambda(n) e \left( \frac{x}{n^\alpha} \right) \ll T^{9/34} N^{12/17} + T^{1/34} N^{16/17} + T^{1/4} N^{2/3} + T^{1/2} + N^{11/12} \label{eq:exp_vdc},
\end{equation}
provided that \eqref{eq:cond_lemma} is fullfiled with $R$ replaced by $N$ and where $T := xN^{- \alpha}$. Set $S(x):=\sum_{t,2}(x)$. Then
$$S(x) = \sum_{\left( t^{-t}x \right)^{\frac{1}{t+1}} < p \leqslant x^{\frac{1}{t+1}+ \epsilon}} \left( \psi \left( \sqrt[t]{\frac{x}{p+1}} \, \right)-\psi \left( \sqrt[t]{\frac{x}{p}} \, \right) \right)$$
and recall that $t \geqslant 1$. First, an usual splitting argument yields
\begin{equation}
  S(x) \ll \max_{\left( t^{-t}x \right)^{\frac{1}{t+1}} < N \leqslant x^{\frac{1}{t+1}+ \epsilon}} \left| \sum_{N < p \leqslant 2N} \left( \psi \left( \sqrt[t]{\frac{x}{p+1}} \, \right)-\psi \left( \sqrt[t]{\frac{x}{p}} \, \right) \right) \right| \log x. \label{eq:step_0}
\end{equation}
Next, for all integers $\left( t^{-t}x \right)^{\frac{1}{t+1}} < N \leqslant x^{\frac{1}{t+1}+ \epsilon}$ and all real numbers $\delta \geqslant 0$,
\begin{equation}
    \sum_{N < p \leqslant 2N} \psi \left( \sqrt[t]{\frac{x}{p+\delta}} \, \right) \ll \frac{1}{\log N} \ \max_{N \leqslant N_1 \leqslant 2N} \sum_{N < n \leqslant N_1} \Lambda(n) \, \psi \left( \sqrt[t]{\frac{x}{n+\delta}} \, \right) + N^{1/2}. \label{eq:step_1}
\end{equation}
Since $\Lambda$ is a real-valued function such that $\Lambda(n) \leqslant \log n \ll n^\epsilon$, we have for all $H \in \Z_{\geqslant 1}$, $\delta \geqslant 0$ and $\epsilon > 0$ that
$$x^{- \epsilon} \sum_{N < n \leqslant N_1} \Lambda(n) \, \psi \left( \sqrt[t]{\frac{x}{n+\delta}} \, \right) \ll \frac{N}{H} + \sum_{h = 1}^H \frac{1}{h} \, \left| \sum_{N < n \leqslant N_1} \Lambda(n) \, e \left( h \sqrt[t]{\frac{x}{n+\delta}} \, \right) \right|.$$
Since $\delta \geqslant 0$, we see that the function $\varphi : u \mapsto \sqrt[t]{\dfrac{x}{u}} - \sqrt[t]{\dfrac{x}{u+\delta}}$ is decreasing, and satisfies $\left| \varphi(u) \right| \leqslant t^{-1} \delta x^{1/t} N^{-1-1/t}$ for all $u \in \left[ N,2N\right]$, so that we derive by Abel summation
\begin{multline}
   x^{- \epsilon} \sum_{N < n \leqslant N_1} \Lambda(n) \, \psi \left( \sqrt[t]{\frac{x}{n+\delta}} \, \right) \ll \frac{N}{H} \\
   + \sum_{h = 1}^H \frac{1}{h} \left( 1 + \frac{h \delta x^{1/t}}{N^{1+1/t}}\right) \max_{N \leqslant N_2 \leqslant N_1} \left| \sum_{N < n \leqslant N_2} \Lambda(n) \, e \left( h \sqrt[t]{\frac{x}{n}}\right) \right|. \label{eq:step_3}
\end{multline}
Now using \eqref{eq:exp_vdc} with $x$ replaced by $h x^{1/t}$ and $\alpha = \frac{1}{t}$ we derive
\begin{multline*}
   x^{-\epsilon} \sum_{N < n \leqslant N_2} \Lambda(n) \, e \left( h \sqrt[t]{\frac{x}{n}}\right) \ll h^{9/34} \left( x^3 N^{8t-3} \right)^{\frac{3}{34t}} +  h^{1/34} \left( x N^{32t-1} \right)^{\frac{1}{34t}} \\
    + h^{1/4} \left( x^3 N^{8t-3} \right)^{\frac{1}{12t}} + h^{1/2}\left( x N^{-1} \right)^{\frac{1}{2t}} + N^{11/12},
\end{multline*}
provided that $N \leqslant c_0 \left( h^t x \right)^{\frac{6}{t+6}}$ to ensure \eqref{eq:cond_lemma}, and inserting this bound in \eqref{eq:step_3} we get
\begin{align*}
   x^{- \epsilon} \sum_{N < n \leqslant N_1} \Lambda(n) \, \psi \left( \sqrt[t]{\frac{x}{n+\delta}} \, \right) &\ll \frac{N}{H} + H^{9/34} \left( x^3 N^{8t-3} \right)^{\frac{3}{34t}} + \delta H^{43/34} \left(x^{43} N^{-10t-43} \right)^{\frac{1}{34t}} \\
   & + H^{1/34} \left( x N^{32t-1} \right)^{\frac{1}{34t}} + \delta H^{35/34} \left(x^{35} N^{-2t-35} \right)^{\frac{1}{34t}} \\
   & + H^{1/4} \left( x^3 N^{8t-3} \right)^{\frac{1}{12t}} + \delta H^{5/4} \left( x^{15} N^{-4t-15} \right)^{\frac{1}{12t}} \\
   & + H^{1/2} \left( x N^{-1} \right)^{\frac{1}{2t}} + \delta H^{3/2} \left( x^3 N^{-2t-3} \right)^{\frac{1}{2t}} \\
   & + N^{11/12} + \delta H \left( x^{12} N^{-t-12} \right)^{\frac{1}{12t}},
\end{align*}
provided that $N \leqslant c_0 x^{\frac{6}{t+6}}$. Now Srinivasan's optimization lemma \cite[Lemma~4]{sri62} yields
\begin{align}
   x^{- \epsilon} \sum_{N < n \leqslant N_1} \Lambda(n) \, \psi \left( \sqrt[t]{\frac{x}{n+\delta}} \, \right) & \ll \left( x^9 N^{33t-9} \right)^{\frac{1}{43t}} + \delta^{34/77} \left(x^{43} N^{33t-43} \right)^{\frac{1}{77t}} \notag  \\
   & \qquad \qquad + \left( x^3 N^{8t-3} \right)^{\frac{3}{34t}} + \delta \left(x^{43} N^{-10t-43} \right)^{\frac{1}{34t}} \notag \\
   & + \left( x N^{33t-1} \right)^{\frac{1}{35t}} + \delta^{34/69} \left( x^{35} N^{33t-35} \right)^{\frac{1}{69t}} \notag \\
   & \qquad \qquad + \left( x N^{32t-1} \right)^{\frac{1}{34t}} + \delta \left(x^{35} N^{-2t-35} \right)^{\frac{1}{34t}} \notag \\
   & + \left( x^{3} N^{11t-3} \right)^{\frac{1}{15t}}  + \delta^{4/9} \left( x^{15} N^{11t-15} \right)^{\frac{1}{27t}} \notag \\
   & \qquad \qquad + \left( x^3 N^{8t-3} \right)^{\frac{1}{12t}} + \delta \left( x^{15} N^{-4t-15} \right)^{\frac{1}{12t}} \notag \\
   & + \left( x N^{t-1} \right)^{\frac{1}{3t}} + \delta^{2/5} \left( x^3 N^{t-3} \right)^{\frac{1}{5t}} \notag \\
   & \qquad \qquad + \left( x N^{-1} \right)^{\frac{1}{2t}} + \delta \left( x^3 N^{-2t-3} \right)^{\frac{1}{2t}} \notag \\
   & + \delta^{1/2} \left( x^{12} N^{11t-12} \right)^{\frac{1}{24t}} \notag \\
   & \qquad \qquad + \delta \left( x^{12} N^{-t-12} \right)^{\frac{1}{12t}} + N^{11/12} \label{eq:final_step},
\end{align}
and the result follows by substituting the estimate \eqref{eq:final_step} into \eqref{eq:step_1} with $\delta=0$ and $\delta=1$ respectively, and then in \eqref{eq:step_0}, yielding \eqref{psi sum} as required.

\section{Proof of Corollary \ref{thm:beatty}}
For the sets based on the Piatestski-Shapiro sequences we infer from \cite{Guo} that
$$\pi(S_{\fl{n^c},x})\sim \frac{(x^c)^{1/c}}{\log x^c}=\frac{x}{c\log x}.$$
Since $\alpha \ge 1$ it follows that $\fl{n^c}\ne \fl{m^c}$ for any $n \ne m$. Therefore
$$\frac{|S_{\fl{n^c},x}|}{\log |S_{\fl{n^c},x}|}\sim\frac{x}{\log x}\not \sim\pi(S_{\fl{n^c},x}).$$

For the sets based on the Beatty sequence we infer from \cite{Guo} that
$$\pi(S_{\fl{\alpha n+ \beta},x})\sim \frac{\alpha x + \beta}{\alpha \log(\alpha x + \beta)}\sim \frac{x}{\log x}.$$
Since $\alpha \ge 1$ it follows that $\fl{\alpha n+ \beta}\ne \fl{\alpha m+ \beta}$ for any $n \ne m$. Therefore
$$\frac{|S_{\fl{\alpha n+ \beta},x}|}{\log |S_{\fl{\alpha n+ \beta},x}|}\sim\frac{x}{\log x}\sim\pi\(S_{\fl{\alpha n+ \beta},x}\).$$
The density is given by
$$\frac{|S_{\fl{\alpha n+ \beta},x}|}{\max{S_{\fl{\alpha n+ \beta},x}}-\min{S_{S_{\fl{\alpha n+ \beta},x}}}}=
\frac{x}{\fl{\alpha x + \beta}-\fl{\alpha+\beta}},$$
which is asymptotically $\frac1{\alpha}$,
concluding the proof.


%
%


\end{document}